\documentclass{amsproc}
\usepackage{amsfonts}
\theoremstyle{plain}

\newtheorem{corollary}{Corollary}

\newtheorem{definition}{Definition}
\newtheorem{example}{Example}

\newtheorem{lemma}{Lemma}

\newtheorem{problem}{Problem}
\newtheorem{proposition}{Proposition}
\newtheorem{remark}{Remark}

\newtheorem{theorem}{Theorem}
\numberwithin{equation}{section}

\begin{document}
\title[Invariance identity and means]{A new invariance identity and means}
\author{Jimmy Devillet}
\curraddr{Mathematics Research Unit, University of Luxembourg, Maison du
Nombre, 6, avenue de la Fonte, L-4364 Esch-sur-Alzette, Luxembourg}
\email{jimmy.devillet@uni.lu}
\author{Janusz Matkowski}
\curraddr{Faculty of Mathematics Computer Science and Econometrics,
Univerity of Zielona G\'{o}ra, Szafrana 4A, PL 65-516 Zielona G\'{o}ra,
Poland}
\email{J.Matkowski@wmie.uz.zgora.pl}

\begin{abstract}
The invariance identity involving three operations $D_{f,g}:X\times
X\rightarrow X$ of the form%
\begin{equation*}
D_{f,g}\left( x,y\right) =\left( f\circ g\right) ^{-1}\left( f\left(
x\right) \oplus g\left( y\right) \right) \text{,}
\end{equation*}%
is proposed. The connections of these operations with means is investigated.
The question when the invariance equality admits three means leads to a
composite functional equation. Problem to determine its continuous solutions
is posed.
\end{abstract}

\maketitle

\section{Introduction}

\footnotetext{\textit{2010 Mathematics Subject Classification. }Primary:
26A18, 26E60, 39B12.
\par
\textit{Keywords and phrases:} invariant functions, mean, invariant mean,
reflexivity, iteration, functional equation
\par
{}}

The invariance of a function with respect to a mapping is frequently
advantageous and helpful. For instance, the existence of a fixpoint of a
mapping can be interpreted as an invariance of a constant function with
respect to this map; the knowledge of the invariant mean with respect to a
mean-type mapping allows to find the (nonconstant) limit of the sequence of
the iterates of this mapping and to solve effectively some functional
equations.

In the present paper, given arbitrary sets $X,Y,$ a binary operation $\oplus
:X\times X\rightarrow X$ and bijective functions $f:X\rightarrow X,$ $%
g:Y\rightarrow X$, we consider the mapping $D_{f,g}:X\times Y\rightarrow Y$
of the form%
\begin{equation*}
D_{f,g}\left( x,y\right) =\left( f\circ g\right) ^{-1}\left( f\left(
x\right) \oplus g\left( y\right) \right) \text{,}
\end{equation*}%
where "$\circ $" stands for the composition (Section 2). In the case when
the operation $\oplus $ is bisymmetric and $Y=X$ (that is, when $(X,\oplus )$
is a \textit{medial groupoid} (\cite{Jez78,Jez83,Jez83(1),Jez94,Kep81})), we
show that for arbitrary bijective functions $f,g,h:X\rightarrow X$ such that $D_{f,g}$ and $D_{g,h}$ are reflexive, the
function $D_{f\circ g,g\circ h}$ is invariant with respect to the mapping $%
\left( D_{f,g},D_{g,h}\right) :X^{2}\rightarrow X^{2},$ i.e. that
\begin{equation*}
D_{f\circ g,g\circ h}\circ \left( D_{f,g},D_{g,h}\right) =D_{f\circ g,g\circ
h}\text{.}
\end{equation*}

The similarity of $D_{f,g}$ to the generalized weighted quasiarithmetic
means introduced in 2003 in \cite[Remark 1]{Matko03} (see also \cite{BaPa09,Matko08,Matko10}) motivates our considerations in Section 3. Assuming
that $X$ is an open real interval $I\subset \mathbb{R},$ and the operation "$%
\oplus "$ is the addition, we examine conditions under which $D_{f,g}$ is a
bivariable mean. The reflexivity property of every mean leads to the
iterative functional equation
\begin{equation*}
f\left( g\left( x\right) \right) =f\left( x\right) +g\left( x\right) \text{,
\ \ \ }x\in I,
\end{equation*}%
where $I=\left( 0,\infty \right) $ (or $I=\left[ 0,\infty \right) $) and the
functions $f$ and $g$ are unknown. Under some natural conditions, Theorem 2,
the main result of this section, says in particular that $D_{f,g}$ is a
bivariable mean in $\left( 0,\infty \right) $, if and only if,
\begin{equation*}
f=\sum_{k=0}^{\infty }g^{-k},
\end{equation*}%
where the series $\sum_{k=0}^{\infty }g^{-k}$ of iterates of $g^{-1}\,\ $%
converges uniformly on compact subsets of $\left( 0,\infty \right) $;
moreover $D_{f,g}=\mathcal{D}_{g}$, where
\begin{equation*}
\mathcal{D}_{g}\left( x,y\right) =\left( \sum_{k=0}^{\infty }g^{-k+1}\right)
^{-1}\left( \sum_{k=0}^{\infty }g^{-k}\left( x\right) +g\left( y\right)
\right) ,\ \ \ \ x,y>0,
\end{equation*}%
which justify the names: \textit{iterative type mean} for $\mathcal{D}_{g}$,
and \textit{iterative generator} of this mean for $g$.

If $D_{f,g}=\mathcal{D}_{g}$ and $D_{g,h}=\mathcal{D}_{h},$ then, in view of
Theorem 1, the function $D_{f\circ g,g\circ h}$\ is invariant with respect
to the mean type mapping $\left( \mathcal{D}_{g},\mathcal{D}_{h}\right) $.

On the other hand $\mathcal{D}_{g}$ (and $\mathcal{D}_{h}$) is a very
special case of \textit{generalized weighted quasiarithmetic mean} $%
M_{\varphi ,\psi }:I^{2}\rightarrow I$, of the form%
\begin{equation*}
M_{\varphi ,\psi }\left( x,y\right) =\left( \varphi +\psi \right)
^{-1}\left( \varphi \left( x\right) +\psi \left( y\right) \right) \text{, \
\ \ \ }x,y\in I\text{,}
\end{equation*}%
where $\varphi ,\psi :I\rightarrow \mathbb{R}$ are continuous, of same type
monotonicity, and such that $\varphi +\psi $ is strictly monotonic in the
interval $I$ (\cite{Matko10}). It is known that if $M_{\varphi ,\psi }$ and $%
M_{\psi ,\gamma }$ are two generalized weighted quasiarithmetic means, then $%
M_{\varphi +\psi ,\psi +\gamma },$ the mean of the same type, is a unique
mean that is invariant with respect to the mean-type mapping $\left(
M_{\varphi ,\psi },M_{\psi ,\gamma }\right) :I^{2}\rightarrow I^{2}$ (\cite%
{Matko14}). \ This leads to natural and equivalent questions: \ is $%
D_{f\circ g,g\circ h}$ a mean; is the operation $D_{f\circ g,g\circ h}$ a
generalized weighted quasiarithmetic mean; or simply can the three
operations $D_{f,g}$, $D_{g,h}$ and $D_{f\circ g,g\circ h}$ be means
simultaneously? In Section 4 we prove that it can happen iff
\begin{equation*}
g=\sum_{i=0}^{\infty }h^{-i}\text{, \ \ \ \ \ \ \ \ \ }f=\sum_{j=0}^{\infty
}\left( \sum_{i=0}^{\infty }h^{-i}\right) ^{-j}\text{, }
\end{equation*}%
and $h$ satisfies "strongly" composite functional equation
\begin{equation*}
\text{\ }\left( \sum_{j=0}^{\infty }\left( \sum_{i=0}^{\infty }h^{-i}\right)
^{-j+1}\right) =\sum_{k=0}^{\infty }\left( \sum_{i=0}^{\infty
}h^{-i+1}\right) ^{-k},
\end{equation*}%
where $i,j,k$ stand for the indices of iterates of the suitable functions
(Theorem 3). We propose as an open problem to find the continuous solutions $%
h:\left( 0,\infty \right) \rightarrow \left( 0,\infty \right) .$ In
illustrative Example 3 (with $h\left( x\right) =w^{-1}x$ and $w\in \left(
0,1\right) )$ we get $D_{f,g}\left( x,y\right) =\left( 1-w\right)
x+wy,$ $D_{g,h}\left( x,y\right) =wx+\left( 1-w\right) ,$ and $%
D_{f\circ g,g\circ h}=2w\left( 1-w\right) A$, where $A$ is the arithmetic
mean, which shows $D_{f\circ g,g\circ h}$ need not be a mean.

%
%
%
%
%

\section{The invariance identity}

We begin this section with the following

\begin{remark}
Let $X$ and $Y$ be nonempty sets and $\oplus :X^{2}\rightarrow X$ a binary
operation on $X$. If $f:X\rightarrow X$, $g:Y\rightarrow X$ are bijective,
i.e., one-to-one and onto, then $D_{f,g}:X\times Y\rightarrow Y$ given by%
\begin{equation*}
D_{f,g}\left( x,y\right) :=\left( f\circ g\right) ^{-1}\left( f\left(
x\right) \oplus g\left( y\right) \right) \text{, \ \ \ \ }x\in X\text{, }%
y\in Y\text{,}
\end{equation*}%
is a correctly defined mapping of the product $X\times Y$ into $Y.$
\end{remark}

Indeed, for arbitrary $\left( x,y\right) \in X\times Y,$ we have $f\left(
x\right) \oplus g\left( y\right) \in X.$ Hence, as $f:X\rightarrow X$ is
one-to-one and onto, we have
\begin{equation*}
f^{-1}\left( f\left( x\right) \oplus g\left( y\right) \right) \in X\text{,}
\end{equation*}%
whence, as $g:Y\rightarrow X$ is one-to-one and onto, we have%
\begin{equation*}
g^{-1}\left( f^{-1}\left( f\left( x\right) \oplus g\left( y\right) \right)
\right) \in Y.
\end{equation*}%
Since $\left( f\circ g\right) ^{-1}=g^{-1}\circ f^{-1},$ it follows that
\begin{equation*}
D_{f,g}\left( x,y\right) =\left( f\circ g\right) ^{-1}\left( f\left(
x\right) \oplus g\left( y\right) \right) =g^{-1}\left( f^{-1}\left( f\left(
x\right) \oplus g\left( y\right) \right) \right) \in Y,
\end{equation*}%
which shows that $D_{f,g}$ is a correctly defined mapping of $X\times Y$ into $%
Y$.

\begin{remark}
Under the conditions of Remark 1, the operation $D_{f,g}$ is symmetric,
i.e., $D_{f,g}\left( x,y\right) =D_{f,g}\left( y,x\right) $ for all $x\in
X\, $, $y\in Y$, iff $Y=X$ and the equality%
\begin{equation*}
\left(f\circ g^{-1}\right)\left( x\right) \oplus y=\left(f\circ g^{-1}\right)\left( y\right) \oplus x%
\text{, \ \ \ \ }x,y\in X,
\end{equation*}%
is satisfied.
\end{remark}

In the sequel, assuming that $Y=X$ in Remark 1 we introduce the following

\begin{definition}
Let $X$ be a nonempty set and $\oplus :X^{2}\rightarrow X$ be a binary
operation. For arbitrary bijective functions $f,g:X\rightarrow X,$ we define
$D_{f,g}:X\times X\rightarrow X$ by the formula%
\begin{equation*}
D_{f,g}\left( x,y\right) :=\left( f\circ g\right) ^{-1}\left( f\left(
x\right) \oplus g\left( y\right) \right) \text{, \ \ \ \ }x,y\in X\text{.}
\end{equation*}
\end{definition}

The main result of this section reads as follows

\begin{theorem}
$\,$ Let $X$ be a nonempty set and let $\oplus :X^{2}\rightarrow X$ be a
bisymmetric binary operation, that is%
\begin{equation*}
\left( u\oplus v\right) \oplus \left( w\oplus z\right) =\left( u\oplus
w\right) \oplus \left( v\oplus z\right) \text{, \ \ \ \ \ \ }u,v,w,z\in X.
\end{equation*}%
If $f,g,h:X\rightarrow X$ are bijections such that
\begin{equation*}
f\left( x\right) \oplus g\left( x\right) =\left( f\circ g\right) \left(
x\right) \text{, \ \ \ \ \ }g\left( x\right) \oplus h\left( x\right) =\left(
g\circ h\right) \left( x\right) \text{, \ \ \ \ \ \ \ }x\in X\text{,}
\end{equation*}%
then
\begin{equation}
D_{f\circ g,g\circ h}\circ \left( D_{f,g},D_{g,h}\right) =D_{f\circ g,g\circ
h},  \tag{1}
\end{equation}%
that is $D_{f\circ g,g\circ h}$ is invariant with respect the mapping $%
\left( D_{f,g},D_{g,h}\right) :X^{2}\rightarrow X^{2}.$
\end{theorem}

\begin{proof}
By the definition of $D_{f,g}$, $D_{g,h}$, $D_{f\circ g,g\circ h}$, and the bisymmetry of $\oplus $, we have, for
all $x,y\in X,$%
\begin{eqnarray*}
&&D_{f\circ g,g\circ h}\circ \left( D_{f,g},D_{g,h}\right) \left( x,y\right)
\\
&=&\left( \left( f\circ g\right) \circ \left( g\circ h\right) \right)
^{-1}\left( \left( f\circ g\right) \left( D_{f,g}\left( x,y\right) \right)
\oplus \left( g\circ h\right) \left( D_{g,h}\left( x,y\right) \right)
\right)  \\
&=&\left( \left( f\circ g\right) \circ \left( g\circ h\right) \right)
^{-1}\left( \left( f\circ g\right) \left( \left( f\circ g\right) ^{-1}\left(
f\left( x\right) \oplus g\left( y\right) \right) \right) \oplus \left(
g\circ h\right) \left( \left( g\circ h\right) ^{-1}\left( g\left( x\right)
\oplus h\left( y\right) \right) \right) \right)  \\
&=&\left( \left( f\circ g\right) \circ \left( g\circ h\right) \right)
^{-1}\left( \left( f\left( x\right) \oplus g\left( y\right) \right) \oplus
\left( g\left( x\right) \oplus h\left( y\right) \right) \right)  \\
&=&\left( \left( f\circ g\right) \circ \left( g\circ h\right) \right)
^{-1}\left( \left( f\left( x\right) \oplus g\left( x\right) \right) \oplus
\left( g\left( y\right) \oplus h\left( y\right) \right) \right)  \\
&=&\left( \left( f\circ g\right) \circ \left( g\circ h\right) \right)
^{-1}\left( \left( f\circ g\right) \left( x\right) \oplus \left( g\circ
h\right) \left( y\right) \right)  \\
&=&D_{f\circ g,g\circ h}\left( x,y\right) ,
\end{eqnarray*}%
which proves the result.
\end{proof}

\section{Operation $D_{f,g}$ and means}

Let $I\subset \mathbb{R}$ be an interval that is closed with respect to the
addition and let $f,g:I\rightarrow I$ be bijective functions. Then, in view of
Definition 1, the two-variable function $D_{f,g}:I^{2}\rightarrow I$ given by%
\begin{equation}
D_{f,g}\left( x,y\right) :=\left( f\circ g\right) ^{-1}\left( f\left(
x\right) +g\left( y\right) \right) \text{, \ \ \ \ }x,y\in I\text{,}  \tag{2}
\end{equation}%
is correctly defined.

In this section we examine when $D_{f,g}$ is a mean. Recall that a function $%
M\colon I^{2}\rightarrow \mathbb{R}$ is said to be a \emph{mean} if it is
internal, i.e.
\begin{equation*}
\min \left( x,y\right) \leq M\left( x,y\right) \leq \max \left( x,y\right)
\text{, \ \ \ \ \ }x,y\in I,
\end{equation*}%
and \emph{strict mean} if it is a mean and these inequalities are sharp for
all $x\neq y$.

If $M$ is a mean in $I,$ then $M\left( J^{2}\right) \subset J$ for every
interval $J\subset I$, in particular $M\colon I^{2}\rightarrow I$; moreover $%
M$ is \emph{reflexive, }i.e.%
\begin{equation*}
M\left( x,x\right) =x\text{, \ \ \ \ \ }x\in I\text{.}
\end{equation*}

\begin{remark}
If $D_{f,g}:I^{2}\rightarrow I$ defined by (2) is symmetric then $g=f+c$ for
some real constant $c$; if moreover $D_{f,g}$ is a mean, then%
\begin{equation*}
D_{f,g}\left( x,y\right) =A\left( x,y\right) \text{, \ \ \ \ }x,y\in I\,%
\text{, \ \ \ \ \ \ }
\end{equation*}
\end{remark}

where $A\left( x,y\right) :=\frac{x+y}{2}$ is the arithmetic mean.

Indeed, the first result (that is easy to verify) and the reflexivity of $%
D_{f,g}$ imply that $2f\left( x\right) +c=f\left( f\left( x\right) +c\right)
$ for all $x\in I$. Hence, by the bijectivity of $f,$ we get $f\left(
x\right) =2x-c$ for all $x\in I$, which implies the remark.

\begin{remark}
If a function $M:I^{2}\rightarrow \mathbb{R}$ is reflexive and (strictly)
increasing in each variable then it is a (strict) mean in $I$.
\end{remark}

Since every mean is reflexive, we first consider conditions for reflexivity
of $D_{f,g}$. We begin with

\begin{lemma}
Let $I$ be a nontrivial interval that is closed with respect to the
addition. Assume that $f,g:I\rightarrow I$ are bijective functions such that
$D_{f,g}$ defined by (2) is reflexive, i.e. that
\begin{equation}
\left( f\circ g\right) ^{-1}\left( f\left( x\right) +g\left( x\right)
\right) =x\text{, \ \ \ \ \ }x\in I.  \tag{3}
\end{equation}%
If $g$ has a fixpoint $x_{0}\in I,$ then $x_{0}=0;$ in particular $0$ must
belong to $I;$

if moreover $g$ is continuous and $I\subset \left[ 0,\infty \right) ,$ then $%
I=\left[ 0,\infty \right) $; $g$ is strictly increasing and either
\begin{equation*}
\text{\ \ \ }0<g\left( x\right) <x\text{, \ \ \ }x\in \left( 0,\infty
\right) \text{, }
\end{equation*}%
\ \ \ \ \ or%
\begin{equation*}
g\left( x\right) >x,\ \ \ \ \ \ x\in \left( 0,\infty \right) .
\end{equation*}%
\
\end{lemma}

\begin{proof}
From (3) we have
\begin{equation*}
f\left( g\left( x\right) \right) =f\left( x\right) +g\left( x\right) \text{,
\ \ \ }x\in I.
\end{equation*}%
Thus, if $g\left( x_{0}\right) =x_{0}$ for some $x_{0}\in I$, then $f\left(
x_{0}\right) =f\left( x_{0}\right) +x_{0}$, so $x_{0}=0.$ Hence, if $%
I\subset \left[ 0,\infty \right) $ then, as $I$ is nontrivial and closed
with respect to addition, it must be of the form $\left[ 0,\infty \right) $.
Since $g$ has no fixpoints in $\left( 0,\infty \right) $, the continuity of $%
g$ implies it must be increasing and either \ $0<g\left( x\right) <x$ \ for
all $x\in I$, or $g\left( x\right) >x$ for all $x\in I$.
\end{proof}

This lemma justifies the assumption that $I=\left( 0,\infty \right) $ in our
considerations of reflexivity of $D_{f,g}$.

\begin{proposition}
Let $g:\left( 0,\infty \right) \rightarrow \left( 0,\infty \right) $ be
injective continuous and such that%
\begin{equation}
0<g\left( x\right) <x\text{, \ \ \ \ \ \ }x>0.  \tag{4}
\end{equation}%
Then there is no continuous function$\ f:\left( 0,\infty \right)
\rightarrow \left( 0,\infty \right) $ satisfying equation
\begin{equation}
f\left( g\left( x\right) \right) =f\left( x\right) +g\left( x\right) \text{,
\ \ \ \ \ }x\in \left( 0,\infty \right) ;  \tag{5}
\end{equation}%
in particular, there is no injective continuous function $f:\left(
0,\infty \right) \rightarrow \left( 0,\infty \right) $ such that $D_{f,g}$
is reflexive in $(0,\infty )$.
\end{proposition}

\begin{proof}
The continuity of $g$ and condition (4) imply that%
\begin{equation*}
\lim_{n\rightarrow \infty }g^{n}\left( x\right) =0\text{, \ \ \ \ \ }x>0,
\end{equation*}%
where $g^{n}$ denotes the $n$th iterate of $g$.

Assume that there is a continuous and strictly increasing function $f:\left(
0,\infty \right) \rightarrow \left( 0,\infty \right) $ satisfying (5).\ From
(5), by induction we get%
\begin{equation*}
f\left( g^{n}\left( x\right) \right) =f\left( x\right)
+\sum_{k=1}^{n}g^{k}\left( x\right) \text{, \ \ \ \ \ \ \ }x\in \left(
0,\infty \right) \text{, \ }n\in \mathbb{N.}
\end{equation*}%
Since $f$ is nonnegative and increasing, it has a finite right-hand side
limit at $0$, denoted by $f\left( 0+\right) $. Letting here $n\rightarrow
\infty $, we obtain%
\begin{equation*}
f\left( 0+\right) =f\left( x\right) +\sum_{k=1}^{\infty }g^{k}\left(
x\right) \text{, \ \ \ \ \ \ \ }x\in \left( 0,\infty \right) ,
\end{equation*}%
that is a contradiction, as the left side is real constant and right side is
either strictly increasing or $\infty $.
\end{proof}

\begin{proposition}
Let $g:\left( 0,\infty \right) \rightarrow \left( 0,\infty \right) $ be
bijective, continuous and such that%
\begin{equation}
g\left( x\right) >x\text{, \ \ \ \ \ \ }x>0.  \tag{6}
\end{equation}%
Then the following conditions are equivalent:

(i) there is a continuous function $f:\left( 0,\infty \right) \rightarrow
\left( 0,\infty \right) $ such that $D_{f,g}$ is reflexive; \

(ii) there is a continuous function$\ f:\left( 0,\infty \right) \rightarrow \left(
0,\infty \right) $ satisfying (5):
\begin{equation*}
f\left( g\left( x\right) \right) =f\left( x\right) +g\left( x\right) ,\ \ \
\ \ x\in \left( 0,\infty \right) ;
\end{equation*}

(iii) there are a function $\ f:\left( 0,\infty \right) \rightarrow \left(
0,\infty \right) $ and $c\geq 0$ such that%
\begin{equation}
f\left( x\right) =c+\sum_{k=0}^{\infty }g^{-k}\left( x\right) \,,\text{ \ \
\ }x\in \left( 0,\infty \right) ,  \tag{7}
\end{equation}%
where $g^{-k}$ denotes the $k$th iterate of the function $g^{-1}$.
\end{proposition}

\begin{proof}
The implication $(i)\Longrightarrow (ii)$ is obvious.

Assume (ii)$.$ The assumptions on $g$ imply that $g^{-1},$ the inverse of $%
g, $ is continuous, strictly increasing, and, in view of (6),
\begin{equation*}
0<g^{-1}\left( x\right) <x\text{, \ \ \ \ \ \ }x>0.
\end{equation*}%
Consequently,
\begin{equation*}
\lim_{n\rightarrow \infty }g^{-n}\left( x\right) =0,\text{ \ \ \ \ \ \ }%
x>0\,,
\end{equation*}%
where $g^{-n}$ stands for the $n$th iterate of $g^{-1}\,.$

Assume that there is a continuous function $f$ such that equality (5) holds.
Replacing $x$ by $g^{-1}\left( x\right) $ in (5) we get%
\begin{equation*}
f\left( g^{-1}\left( x\right) \right) =f\left( x\right) -x\text{, \ \ \ \ \ }%
x\in \left( 0,\infty \right) ,
\end{equation*}%
whence, by induction,
\begin{equation*}
f\left( g^{-n}\left( x\right) \right) =f\left( x\right)
-\sum_{k=0}^{n-1}g^{-k}\left( x\right) \text{, \ \ \ \ \ \ \ }x\in \left(
0,\infty \right) ,\text{ }n\in \mathbb{N}\text{.}
\end{equation*}%
Since $c:=f\left( 0+\right) \geq 0$ exists, letting here $n\rightarrow
\infty ,$ we obtain
\begin{equation*}
c=f\left( x\right) -\sum_{k=0}^{\infty }g^{-k}\left( x\right) \,,\text{ \ \
\ }x\in \left( 0,\infty \right) ,
\end{equation*}%
whence%
\begin{equation*}
f\left( x\right) =c+\sum_{k=0}^{\infty }g^{-k}\left( x\right) \,,\text{ \ \
\ }x\in \left( 0,\infty \right) ,
\end{equation*}%
where the series $\sum_{k=1}^{\infty }g^{-k}$ converges, and its sum is a
continuous function, which proves that (iii) holds.

Assume (iii). If $f:\left( 0,\infty \right) \rightarrow \left( 0,\infty
\right) ~$is of the form (7) then, for every constant $c\geq 0$ and $x\in
\left( 0,\infty \right) $\ we have%
\begin{eqnarray*}
f\left( g\left( x\right) \right) &=&c+\sum_{k=0}^{\infty }g^{-k}\left(
g\left( x\right) \right) =c+g\left( x\right) +\sum_{k=1}^{\infty
}g^{-k}\left( g\left( x\right) \right) \\
&=&\left( c+\sum_{k=1}^{\infty }g^{-k}\left( g^{1}\left( x\right) \right)
\right) +g\left( x\right) =\left( c+\sum_{k=0}^{\infty }g^{-k}\left(
x\right) \right) +g\left( x\right) \\
&=&f\left( x\right) +g\left( x\right) .
\end{eqnarray*}%
which implies that $D_{f,g}$ is reflexive, so (i) holds.
\end{proof}

\bigskip Now we prove the main result of this section.

\begin{theorem}
Let $f,g:\left( 0,\infty \right) \rightarrow \left( 0,\infty \right) $ be
continuous, strictly increasing, and onto. The following conditions are
equivalent:

(i) the function $D_{f,g}$ is a bivariable strict mean in $\left(
0,\infty \right) ;$

(ii) the function $D_{f,g}$ is reflexive;

(iii) the series $\sum_{k=0}^{\infty }g^{-k}$ of iterates of $g^{-1}\,\ $%
converges uniformly on compact subsets of $\left( 0,\infty \right) $ and,
\begin{equation}
f=\sum_{k=0}^{\infty }g^{-k};  \tag{8}
\end{equation}%
moreover $D_{f,g}=\mathcal{D}_{g},$ where $\mathcal{D}_{g}:\left( 0,\infty
\right) ^{2}\rightarrow \left( 0,\infty \right) $ defined by
\begin{equation}
\mathcal{D}_{g}\left( x,y\right) =\left( \sum_{k=0}^{\infty }g^{-k+1}\right)
^{-1}\left( \sum_{k=0}^{\infty }g^{-k}\left( x\right) +g\left( y\right)
\right) \text{, \ \ \ \ }x,y>0,  \tag{9}
\end{equation}%
is a strictly increasing bivariable mean in $\left( 0,\infty \right) $.
\end{theorem}

\begin{proof}
It is obvious that (i) implies (ii). Assume that (ii) holds. Then, applying
Propositions 1 and 2 and taking into account that $c:=f\left( 0+\right) =0$
(by the bijectivity of $f$) we get the first part of (iii); in particular
the uniform convergence on compact subsets of the series of iterates follows
from (7) and the Dini theorem. The "moreover" result follows immediately
from (8) and the definition of $D_{f,g}$. Thus (ii) implies (iii).

To prove (i) assuming (iii), note that the function $\mathcal{D}_{g}$ defined by (9) is
reflexive and strictly increasing in each variable so, by Remark 4, it is a
strict mean.
\end{proof}

From the above theorem we obtain the following

\begin{corollary}
Let a continuous strictly increasing function $r:\left( 0,\infty \right)
\rightarrow \left( 0,\infty \right) $ be such that%
\begin{equation*}
0<r\left( x\right) <x\text{, \ \ \ \ }x>0,
\end{equation*}%
and the series
\begin{equation*}
\sum_{k=0}^{\infty }r^{k}
\end{equation*}%
converges to a finite continuous function. Then the function $\mathcal{D}%
_{r}:\left( 0,\infty \right) ^{2}\rightarrow \left( 0,\infty \right) $
defined by
\begin{equation*}
\mathcal{D}_{r}\left( x,y\right) :=\left( \sum_{k=0}^{\infty }r^{k-1}\right)
^{-1}\left( \sum_{k=0}^{\infty }r^{k}\left( x\right) +r^{-1}\left( y\right)
\right) \text{, \ \ \ \ }x,y>0
\end{equation*}%
is a bivariable strict mean in $\left( 0,\infty \right) $.
\end{corollary}

\begin{proof}
It is enough to apply the previous result with $g:=r^{-1}$, \ $%
f:=\sum_{k=0}^{\infty }r^{k}$ and observe that $\mathcal{D}_{r}=D_{f,g}$.
\end{proof}

Since the mean $\mathcal{D}_{r}$ is strictly related to the iterates of $r$,
we propose the following

\begin{definition}
If $r:\left( 0,\infty \right) \rightarrow \left( 0,\infty \right) $
satisfies the assumptions of Corollary 1, then the function $\mathcal{D}_{r}$
is called an iterative mean of generator $r$.
\end{definition}

\begin{example}
Applying the definition of $\mathcal{D}_{r}$ for the generator $r:\left(
0,\infty \right) \rightarrow \left( 0,\infty \right) $ given by $r\left(
x\right) =wx$, where $w\in \left( 0,1\right) $ is fixed, we get weighted
arithmetic mean $\mathcal{D}_{r}\left( x,y\right) =$ $wx+\left( 1-w\right) y$
\ for all $x,y\in \left( 0,\infty \right) .$ \ \ \
\end{example}

\begin{example}
Let $p\in \left( 0,1\right] $. The function $r:\left( 0,\infty \right)
\rightarrow \left( 0,\infty \right) $ given by
\begin{equation*}
r\left( x\right) =\frac{px^{2}}{x+1}\text{, \ \ \ \ }x>0,
\end{equation*}%
is strictly increasing, and we have
\begin{equation*}
0<r\left( x\right) <px\text{, \ \ \ \ }x>0.
\end{equation*}%
It follows that in the case when $p<1$, the series $\sum_{k=0}^{\infty
}r^{k} $ converges uniformly on compact sets, and, consequently, the
function $r$ generates the iterative mean $\mathcal{D}_{r}$.\
\end{example}

\begin{remark}
In Theorem 2 we assume that $I=\left( 0,\infty \right) .$ One could also get
the suitable results if $I$ is one the following intervals $\left[ 0,\infty
\right) $, $\left( -\infty ,0\right) ,$ $\left( -\infty ,0\right] ,$ as well
as $\mathbb{R}$.
\end{remark}

\section{Invariant operation with respect to iterative mean-type mapping
need not be a mean}

In this section we focus our attention on the question whether the invariant
function occurring in (1):%
\begin{equation*}
D_{f\circ g,g\circ h}\circ \left( D_{f,g},D_{g,h}\right) =D_{f\circ g,g\circ
h},
\end{equation*}%
is a mean, if the coordinates of the mapping $\left( D_{f,g},D_{g,h}\right) $
defined by (2) are means.

If $D_{f,g}=\mathcal{D}_{g}$ and $D_{g,h}=\mathcal{D}_{h},$ then, in view of
Theorem 1, the function $D_{f\circ g,g\circ h}$\ is invariant with respect
to the mean type mapping $\left( \mathcal{D}_{g},\mathcal{D}_{h}\right) $.

On the other hand $\mathcal{D}_{g}$ (and $\mathcal{D}_{h}$) is a very
special case of generalized weighted quasiarithmetic mean $M_{\varphi ,\psi
}:I^{2}\rightarrow I$,%
\begin{equation*}
M_{\varphi ,\psi }\left( x,y\right) =\left( \varphi +\psi \right)
^{-1}\left( \varphi \left( x\right) +\psi \left( y\right) \right) \text{, \
\ \ \ }x,y\in I\text{,}
\end{equation*}%
where $\varphi ,\psi :I\rightarrow \mathbb{R}$ are continuous, of same type
monotonicity, and such that $\varphi +\psi $ is strictly monotonic in the
interval $I$ (\cite{Matko10}). Indeed, with $\varphi =\sum_{k=0}^{\infty
}g^{-k}$ and $\psi =g$ we have $M_{\varphi ,\psi }=\mathcal{D}_{g}$ (and
with $\psi =\sum_{k=0}^{\infty }h^{-k}$ and $\gamma =h$ we have $M_{\psi
,\gamma }=\mathcal{D}_{h}$). Moreover, it is known that if $M_{\varphi ,\psi
}$ and $M_{\psi ,\gamma }$ are two generalized quasiarithmetic means, then $%
M_{\varphi +\psi ,\psi +\gamma },$ the mean of the same type, is a unique
mean that is invariant with respect to the mean-type mapping $\left(
M_{\varphi ,\psi },M_{\psi ,\gamma }\right) :I^{2}\rightarrow I^{2}$ (\cite%
{Matko13}).

In this context the question arises whether the invariant function $D_{f\circ
g,g\circ h}$ coincides with the invariant mean $M_{\varphi +\psi ,\psi
+\gamma }?$ In view of (\cite{Matko13}), the answer is
positive, if $D_{f\circ g,g\circ h}$ is a mean, i.e. if there is a suitable
bijective function $u:\left( 0,\infty \right) \rightarrow \left( 0,\infty
\right) $ such that $D_{f\circ g,g\circ h}=\mathcal{D}_{u}$.

We prove

\begin{theorem}
Assume that $f,g,h\colon (0,\infty) \to (0,\infty)$ are continuous, strictly increasing, and onto. Then $D_{f,g}$, $D_{g,h}$, and
$D_{f\circ g,g\circ h}$\ are means in $\left( 0,\infty \right) , $ iff
\begin{equation}
g=\sum_{i=0}^{\infty }h^{-i}\text{, \ \ \ \ \ \ \ \ \ }f=\sum_{j=0}^{\infty
}\left( \sum_{i=0}^{\infty }h^{-i}\right) ^{-j}\text{, }  \tag{10}
\end{equation}%
and $h$ satisfies the composite functional equation
\begin{equation}
\text{\ } \sum_{j=0}^{\infty }\left( \sum_{i=0}^{\infty }h^{-i}\right)
^{-j+1} =\sum_{k=0}^{\infty }\left( \sum_{i=0}^{\infty }h^{-i+1}\right)
^{-k},  \tag{11}
\end{equation}%
where $i,j,k$ stand for the indices of iterates of the suitable functions.
\end{theorem}

\begin{proof}
Assume that $D_{f,g}$, $D_{g,h}$ and $D_{f\circ g,g\circ h}$\ are means in $%
\left( 0,\infty \right) .$ In view of Theorem 2, we have $D_{g,h}=\mathcal{D}%
_{h}$, $D_{f,g}=\mathcal{D}_{g},$ $D_{f\circ g,g\circ h}=\mathcal{D}_{g\circ
h}$, and
\begin{equation}
g=\sum_{i=0}^{\infty }h^{-i}\text{, \ \ \ \ \ \ }f=\sum_{j=0}^{\infty }g^{-j}%
\text{, \ \ \ \ \ \ \ }f\circ g=\sum_{k=0}^{\infty }\left( g\circ h\right)
^{-k}.  \tag{12}
\end{equation}%
It follows that (10) holds true, and the third of equalities (12) implies
that $h$ satisfies the composite functional equation,
\begin{equation*}
\text{\ }\left( \sum_{j=0}^{\infty }\left( \sum_{i=0}^{\infty }h^{-i}\right)
^{-j}\right) \circ \left( \sum_{i=0}^{\infty }h^{-i}\right)
=\sum_{k=0}^{\infty }\left( \left( \sum_{i=0}^{\infty }h^{-i}\right) \circ
h\right) ^{-k},
\end{equation*}%
which simplifies to (11). The converse implication is obvious.
\end{proof}

Thus our question leads to rather complicate composite functional equation
(11). We pose the following

\begin{problem}
Determine strictly increasing bijective functions $h:\left( 0,\infty \right)
\rightarrow \left( 0,\infty \right) $ satisfying equation (11).
\end{problem}

Let us consider the following

\begin{example}
Take $w\in \left( 0,1\right) ,$ and define $h:\left( 0,\infty \right)
\rightarrow \left( 0,\infty \right) $ by $h\left( x\right) =\frac{x}{w}$.
From Theorem 2 with $g$ and $f$ replaced, respectively, by $h$ and $g,$ we
get%
\begin{equation*}
g\left( x\right) =\sum_{k=0}^{\infty }h^{-k}\left( x\right)
=\sum_{k=0}^{\infty }w^{k}x=\frac{x}{1-w}\text{, \ \ \ \ \ }x\in \left(
0,\infty \right) ,
\end{equation*}%
and, for all $x,y>0$,
\begin{eqnarray*}
\mathcal{D}_{h}\left( x,y\right) &=& \left( \sum_{k=0}^{\infty
}h^{-k+1}\right) ^{-1}\left( \sum_{k=0}^{\infty }h^{-k}\left( x\right)
+h\left( y\right) \right) \\
&=& w\left( 1-w\right) \left( \frac{x}{1-w}+\frac{y}{w}\right) =wx+\left(
1-w\right) y.
\end{eqnarray*}%
Similarly, as $g^{-1}\left( x\right) =\left( 1-w\right) x,$ we get
\begin{equation*}
f\left( x\right) =\sum_{k=0}^{\infty }g^{-k}\left( x\right)
=\sum_{k=0}^{\infty }\left( 1-w\right) ^{k}x=\frac{x}{w}\text{, \ \ \ \ \ }%
x\in \left( 0,\infty \right) ,
\end{equation*}%
and, for all $x,y>0,$%
\begin{eqnarray*}
\mathcal{D}_{g}\left( x,y\right) &=& \left( \sum_{k=0}^{\infty
}g^{-k+1}\right) ^{-1}\left( \sum_{k=0}^{\infty }g^{-k}\left( x\right)
+g\left( y\right) \right) \\
&=& \left( 1-w\right) w\left( \frac{1}{w}x+\frac{y}{1-w}\right) =\left(
1-w\right) x+wy.
\end{eqnarray*}%
Moreover, since $\left( f\circ g\right) \left( x\right) =\frac{x}{w\left(
1-w\right) }$, \ $\left( g\circ h\right) \left( x\right) =\frac{x}{\left(
1-w\right) w}$, \ we get, for all $x,y>0,$%
\begin{eqnarray*}
D_{f\circ g,g\circ h}\left( x,y\right) &=&\left( \left( f\circ g\right)
\circ \left( g,h\right) \right) ^{-1}\left( \left( f\circ g\right) \left(
x\right) +\left( g\circ h\right) \left( y\right) \right) \\
&=& \left( w\left( 1-w\right) \right) ^{2}\left( \frac{x}{w\left( 1-w\right)
}+\frac{y}{\left( 1-w\right) w}\right) \\
&=&w\left( 1-w\right) \left( x+y\right) ,
\end{eqnarray*}%
whence
\begin{equation*}
D_{f\circ g,g\circ h}=2w\left( 1-w\right) A,
\end{equation*}%
where $A\left( x,y\right) =\frac{x+y}{2}$ ($x,y>0$) is the arithmetic mean.

For every $w\in \left( 0,1\right) $ we have $A\circ \left( \mathcal{D}_{g},%
\mathcal{D}_{h}\right) =A\,,$ and $A$ is a unique $\left( \mathcal{D}_{g},%
\mathcal{D}_{h}\right) $-invariant mean. The function $D_{f\circ g,g\circ h}$
is also invariant for every $w\in \left( 0,1\right) $, but $D_{f\circ
g,g\circ h}$ is not a mean for any $w\in \left( 0,1\right) $.
\end{example}

In the context of the above discussion let us consider the following

\begin{remark}
Under conditions of Theorem 3, the bijective functions $f,g,h:\left(
0,\infty \right) \rightarrow \left( 0,\infty \right) $ are increasing.
Consequently, they are almost everywhere differentiable and $f\left(
0+\right) =g\left( 0+\right) =h\left( 0+\right) =0$. Assume additionally
that $f\left( 0\right) =g\left( 0\right) =h\left( 0\right) =0$ and that $%
f,g,h$ are differentiable at the point $0$. If the functions $D_{f,g}$, $%
D_{g,h}$ and $D_{f\circ g,g\circ h}$\ were means in $\left[ 0,\infty \right)
$ then, by their reflexivity, we would have%
\begin{equation*}
f\left( g\left( x\right) \right) =f\left( x\right) +g\left( x\right) \text{,
\ }g\left( h\left( x\right) \right) =g\left( x\right) +h\left( x\right)
\text{, \ }f\left( g^{2}\left( h\left( x\right) \right) \right) =f\left(
g\left( x\right) \right) +g\left( h\left( x\right) \right) \text{, \ }
\end{equation*}%
for all $x\geq 0$, where $g^{2}$ is the second iterate of $g$.
Differentiating both sides of each of these equalities at $x=0$ and then setting%
\begin{equation*}
a:=f^{\prime }\left( 0\right) ,\text{ \ \ \ \ }b:=g^{\prime }\left( 0\right)
\text{, \ \ \ \ \ }c:=h^{\prime }\left( 0\right) ,
\end{equation*}%
we would get%
\begin{equation*}
ab=a+b\text{, \ \ \ \ }bc=b+c\text{, \ \ \ \ }ab^{2}c=ab+bc.
\end{equation*}%
Since this system has no solution satisfying the condition $a\geq 1,b\geq
1,c\geq 1$, the functions $f,g,h$ do not exist.
\end{remark}


\section{\protect\bigskip Final remarks}

\begin{remark}
(An extension of Theorem 1) Let $(Y,\oplus )$ be a medial groupoid, $%
X\subseteq Y$ be a subset, and $\diamond \colon (Y^{X})^{2}\rightarrow Y^{X}$
be an operation (with the convenient notation $f\diamond g$ instead of $%
\diamond (f,g),$ for $f,g\in Y^{X}$).

Let $f,g\in Y^{X}$ such that $f\diamond g$ is invertible and $Ran(f\diamond
g)=\{f(x)\oplus g(y)\mid x,y\in X\}$. The operation $A_{f,g}\colon
X^{2}\rightarrow X$  given by
\begin{equation*}
A_{f,g}(x,y)=(f\diamond g)^{-1}(f(x)\oplus g(y)),\quad x,y\in X,
\end{equation*}%
is well-defined.

Similar argument to the one of Theorem 1 shows that: \emph{for all }$%
f,g,h\in Y^{X}$\emph{\ such that the pairs }$\left( f,g\right) ,\left(
g,h\right), $ $\left( f\diamond g,g\diamond h\right) $\emph{\
satisfy the above conditions, and $A_{f,g}$, $A_{g,h}$ are reflexive, we have}\textit{\ }%
\begin{equation*}
A_{f\diamond g,g\diamond h}\circ \left( A_{f,g},A_{g,h}\right) =A_{f\diamond
g,g\diamond h}.
\end{equation*}
\end{remark}


%
%
%

\begin{remark}
Replacing addition by multiplication in the definition of $D_{f,g}$, we can
define another operation
\begin{equation*}
C_{f,g}\left( x,y\right) =\left( f\circ g\right) ^{-1}\left( f\left(
x\right) g\left( y\right) \right) \text{, \ \ \ \ }x,y\in I,
\end{equation*}%
and consider analogous questions.  $\,$
\end{remark}

\section*{Acknowledgements}

The authors thank the reviewer for his/her quick and careful review and useful suggestions. This research is partly supported by the
internal research project R-AGR-0500 of the University of Luxembourg
and by the Luxembourg National Research Fund R-AGR-3080.
%

\bigskip

\end{document}